\documentclass[10pt]{amsart}

\usepackage{amsmath,amsfonts,amssymb, color,cmll}
\usepackage{tikz}
\usepackage{caption} 
\usepackage{enumerate}

\theoremstyle{plain}
\newtheorem{thm}{Theorem}[section]
\newtheorem{lem}[thm]{Lemma}
\newtheorem{corollary}[thm]{Corollary}
\newtheorem{prop}[thm]{Proposition}

\newtheorem*{claim*}{Claim}
\newtheorem*{thm*}{Theorem}

\newtheorem{remark}[thm]{Remark}

\def\R{\mathbb{R}}

\def\Q{\mathbb{Q}}
\def\Z{\mathbb{Z}}

\def\M{\mathcal{M}}
\def\V{\mathcal{V}}
\def\H{\mathcal{H}}

\newcommand{\supp}{\operatorname{supp}}

\newcommand{\Span}{\operatorname{Span}}
\begin{document}
\title{On the positive Lebesgue measure of $(n,k,t)$-sets}
\author{Longhui Li}
\address{Department of Mathematics, Southern
University of Science and Technology, Shenzhen, 518055, PR China}
\email{12231267@mail.sustech.edu.cn}

\begin{abstract}
We study sets in $\R^{n}$ that contain a unit $k$-dimensional disk in
each direction from a subset of the Grassmannian $G(n,k)$ of Hausdorff
dimension at least $t$; we call them $(n,k,t)$-sets. In the range $2k>n$, we prove that every
$(n,k,t)$-set has positive Lebesgue measure whenever $t>k(n-k-1)+1$. This follows from a sharp Sobolev to $L^{2}$ estimate for the
$k$-plane maximal operator over a Frostman measure on $G(n,k)$, which also
yields an $L^{p}$--$L^{2}$ estimate for every $ p>n/k$ with respect to Haar measure. In the range $2k\le n$, where translations can no longer be ignored, we obtain the
same threshold via a lifting argument, conditional on the quantitative
Kakeya estimate in dimension $n-k+1$. Using the quantitative Kakeya estimate of Wang and Zahl in dimension 3, we deduce
that $(4,2,t)$-sets have positive measure for $t>3$.
The threshold $t>k(n-k-1)+1$ is sharp, in the
sense that $(n,k,t)$-sets of measure zero exist at $t=k(n-k-1)+1$.  These
results improve the previous threshold $t>n/2+k(n-k-1)$ of Fraser, Harris and
Kroon.

\end{abstract}
\maketitle
\section{introduction}
\subsection{$(n,k)$-sets}
A Kakeya set is a compact set $K\subset\R^n$ that contains
a unit line segment in every direction. Besicovitch \cite{Bes1928} showed that Kakeya sets of Lebesgue measure zero exist for $n\geq2$. 

It is natural to ask similar questions in higher dimensions, where we replace unit line segments by unit disks. Let $G(n,k)$ be the Grassmannian, i.e., the set of $k$-dimensional subspaces in $\R^n$ equipped with the metric
\begin{align*}
    d(V,W)=\|\pi_V-\pi_W\|,\;V,W\in G(n,k),
\end{align*}
where $\pi_V$ is the orthogonal projection onto $V$ and $\|\cdot\|$ is the operator norm.  

Suppose $k,n$ are positive integers with $k<n$. We say that $B\subset \R^n$ is an $(n,k)$-set if $B$ contains a translation of the unit disk of every element in $G(n,k)$. Kakeya sets are $(n,1)$-sets, and Besicovitch \cite{Bes1928} constructed Kakeya sets of measure $0$. The main conjecture is that $(n,k)$-sets have positive Lebesgue measure for $k\geq 2$.
Marstrand \cite{Mar79} showed that $(3,2)$-sets have positive Lebesgue measure, and Falconer \cite{Fal80} showed that $(n,k)$-sets have positive Lebesgue measure for $k>n/2$. There are many other proofs for Falconer's result \cite{Fal80}, one can see, e.g.,  Peres-Schlag \cite{PS00}, Fraser-Harris-Kroon \cite{Fra22}, and Di Plinio-Parissis \cite{PP22}.

For the range $k\leq n/2$, Bourgain \cite{Bou91} improved Falconer's result to $2^{k-1}+k\geq n$. Oberlin \cite{Obe05} then gave a simple proof of Bourgain's result. Later, Oberlin \cite{Obe10} improved it to $(1+\sqrt{2})^{k-1}+k>n$, which is the best result so far.

Over the $p$-adic field $\mathbb{Q}_p$, by Kakeya maximal estimates for $(\Z/(N\Z))^n$ proven in \cite{Dha22} and the lifting argument given by Oberlin \cite{Obe05}, Dhar \cite{Dha23} showed that $(n,k)$-sets in $\Q_p^n$ have positive measure when $k\geq 2$.

\subsection{$(n,k,t)$-sets} 

We know that $(n,k)$-sets have positive Lebesgue measure for $k\gtrsim \log n$, and the problem becomes easier as $k$
increases. So if $k$ is large, we do not need to assume that a set $B$ contains a translation of the unit disk of \textbf{every} element in $G(n,k)$ to ensure $B$ has positive Lebesgue measure.  We can restrict to a smaller subset of $G(n,k)$.

Suppose $k,n$ are positive integers and  $\V\subset G(n,k)$. We say that $B\subset \R^n$ is an $(n,k,\V)$-set if $B$ contains a translation of the unit disk of every element in $\mathcal{V}$.

It is natural to ask under what conditions on $\mathcal{V}$ $(n,k,\V)$-sets can have positive Lebesgue measure.

 In \cite{Fra22}, Fraser-Harris-Kroon consider an abstract scaling condition and obtain some results. In this article, we focus on the Hausdorff dimension of $\mathcal{V}$ and find a lower bound on the Hausdorff dimension of $\mathcal{V}$ for which $(n,k,\V)$-sets have positive Lebesgue measure. 

Let $\H^\alpha$ be the $\alpha$-dimensional Hausdorff measure, and $\H^\alpha_{\infty}$ be the $\alpha$-dimensional Hausdorff content. When $\alpha$ is the dimension of the ambient space, $\H^\alpha$ is just the Lebesgue measure, and we also use the notation $|\cdot|$. We use the notation $\dim A$ to denote the Hausdorff dimension of the set $A$. Since this article uses only Hausdorff dimension, no confusion should arise. 

We can now define $(n,k,t)$-sets. Suppose $k,n$ are positive integers. We say that $B\subset \R^n$ is an $(n,k,t)$-set if there exists $\V\subset G(n,k)$ with $\dim \V\geq t$ and $B$ contains a translation of the unit disk of every element in $\mathcal{V}$.

     D. M. Oberlin \cite{DMObe07} showed that $(n,n-1,t)$-sets have positive measure if $t>1$. This dimensional exponent is in fact sharp. Recently, Fraser-Harris-Kroon \cite{Fra22} showed that $(n,k,t)$-sets have positive Lebesgue measure for $t> n/2+k(n-k-1)$ and $2k>n$.  One can also follow Falconer's idea \cite{Fal80} to obtain this result. 

\subsection{Maximal Operators} To show the positive measure of $(n,k,t)$-sets, we consider the maximal operators.
    Suppose $f$ is a measurable function with bounded support. For $\delta>0$, we define
\begin{align}
    \mathcal{M}_{\delta}^{n,k}(f)(V)=\frac{1}{\delta^{n-k}}\cdot\sup_{x\in \R^n}\int_{\mathcal{N}_{\delta}(x+V)}|f|, \;V\in G(n,k)
\end{align} 
and
    \begin{align}
    \mathcal{M}^{n,k}(f)(V)=\sup_{x\in \R^n}\int_{x+V}|f|d\mathcal{H}^k, \;V\in G(n,k).
\end{align}

When $k=1$, the operator $\mathcal{M}^{n,1}$
is unbounded because Kakeya sets of zero Lebesgue measure exist, and
$\mathcal{M}_{\delta}^{n,1}$ is the well-known Kakeya maximal operator and is related to the Hausdorff dimension of  Kakeya sets. The main conjecture is that for any $\epsilon>0$,
\begin{align}\label{Kakeya maximal conjecture}
    \|\mathcal{M}_{\delta}^{n,1}(f)\|_{L^p(G(n,1))}\lesssim_{\epsilon} \delta^{1-\frac{n}{p}-\epsilon}\|f\|_{L^p(\R^n)},\;1<p\leq n.
\end{align}
C\'ordoba \cite{Cor77} obtained a sharp $L^2$-estimate in the plane
\begin{align*}
    \|\mathcal{M}_{\delta}^{2,1}(f)\|_{L^2(S^1)}\lesssim \log(1/\delta)^{1/2}\|f\|_{L^2(\R^2)}. 
\end{align*}
We refer the reader to \cite{JKR22} for an introduction to results on the Kakeya maximal function $\M^{n,1}_{\delta}$ when $n\geq 3$.

When $k\geq 2$, the main conjectures, posed by Oberlin \cite{Obe05}, are  
\begin{align}\label{maximal-conjecture}
    \|\mathcal{M}^{n,k}(f)\|_{L^q(G(n,k))}\lesssim\|f\|_{L^p(\R^n)},\;p>n/k,\;1\leq q\leq kp.
\end{align}
and 
\begin{align}
    \|\mathcal{M}_{\delta}^{n,k}(f)\|_{L^q(G(n,k))}\lesssim\delta^{k-n/p}\|f\|_{L^p(\R^n)},\;p<n/k,\;1\leq q\leq (n-k)p'.
\end{align}

For $\M^{n,k}$, Falconer \cite{Fal80} showed that 
\begin{align*}
    \|\mathcal{M}^{n,k}(f)\|_{L^1(G(n,k))}\lesssim\|f\|_{L^p(\R^n)},\;p>n/k,\;k>n/2.
\end{align*}
Later, Bourgain \cite{Bou91} showed that $\M^{4,2}$ is $L^p$-bounded for $p>2$ and  $\M^{7,3}$ is $L^p$-bounded for 
$p>3$. Oberlin \cite{Obe07} showed that if $2^{k-1}+k>n$ and $k\geq 3$,
\begin{align*}
    \|\mathcal{M}^{n,k}(f)\|_{L^{\frac{n-1}{2}}(G(n,k))}\lesssim\|f\|_{L^{\frac{n-1}{2}}(\R^n)}
\end{align*}
and if $(1+\sqrt{2})^{k-1}+k>n$, there exists a large $p=p(n,k)$ such that 
\begin{align*}
    \|\mathcal{M}^{n,k}(f)\|_{L^{p}(G(n,k))}\lesssim\|f\|_{L^{p}(\R^n)}.
\end{align*}

For $\M^{n,k}_{\delta}$, Oberlin \cite{Obe07} showed that for any $\epsilon>0$, there exists a large $p=p(n,k,\epsilon)$ such that 
\begin{align*}
    \|\mathcal{M}_{\delta}^{n,k}(f)\|_{L^{p}(G(n,k))}\lesssim\delta^{-\frac{\alpha+\epsilon}{p}}\|f\|_{L^{p}(\R^n)},
\end{align*}
where $\alpha=\alpha(n,k)=\frac{n-k}{(1+\sqrt{2})^k}$.
Recently, Di Plinio and Parissis \cite{PP22} showed that 
\begin{align*}
     \|\mathcal{M}_{\delta}^{n,k}(f)\|_{L^2(G(n,k))}\lesssim 
      \begin{cases}
        \|f\|_{L^2(\R^n)},\;2k>n;\\
        \log(1/\delta)\|f\|_{L^2(\R^n)},\;2k=n;\\
        \delta^{k-n/2}\|f\|_{L^2(\R^n)},\;2k<n.
    \end{cases} 
   \end{align*}

In the $p$-adic field $\Q^n_p$, Dhar \cite{Dha23} showed that 
\begin{align*}
    \|\mathcal{M}^{n,2}(f)\|_{L^{n-1}(G(\Q^n_p,2))}\lesssim \|f\|_{L^{n-1}(\Q_p^n)}.
\end{align*}

\subsection{$2k>n$}
In the range $2k>n$, our main theorem is the following estimate over fractal measures. 
\begin{thm}\label{maximal-estimate-delta-free}
    Suppose $2k>n$. Let $f\geq 0$ be a measurable function with $\supp f\subset B(0,1)$ and let $\mu$ be a $t$-dimensional Frostman measure on $G(n,k)$ with $t>k(n-k-1)+1$. For any $\epsilon>0$, we have
     \begin{align}
     \|\mathcal{M}^{n,k}(f)\|_{L^2(\mu)}
     \lesssim_{n,k,\mu,t,\epsilon}
    \begin{cases}
        \|f\|_{H^{(k(n-k-1)+1-t)/2+\epsilon}(\R^n)},\;t\leq(k-1)(n-k)+1;\\
        \|f\|_{H^{n/2-k+\epsilon}(\R^{n})},\;t\geq(k-1)(n-k)+1.
    \end{cases}.
   \end{align}
\end{thm}   
The Sobolev exponents in Theorem \ref{maximal-estimate-delta-free}
are sharp up to the $\epsilon$-loss; see Section \ref{sharpness-of-1.1}.
In particular, we have the following result.
\begin{thm}\label{L^2}
    Suppose $2k>n$. Let $f$ be a measurable function with $\supp f\subset B(0,1)$, and let $\mu$ be a $t$-dimensional Frostman measure with $t>k(n-k-1)+1$. We have 
    \begin{align}
        \|\mathcal{M}^{n,k}(f)\|_{L^2(\mu)}
     \lesssim_{n,k,\mu,t}\|f\|_{L^2(\R^n)}.
    \end{align}
    As a corollary, if $2k>n$, $(n,k,t)$-sets have positive Lebesgue measure if $t>k(n-k-1)+1$.

\end{thm}

The range $t>k(n-k-1)+1$ in Theorem \ref{L^2} is sharp because $(n,k,k(n-k-1)+1)$-sets of measure zero exist. Indeed, $(n,k,k(n-k-1)+1)$-sets of zero Lebesgue measure exist for all $k\geq 2$. There are two kinds of such examples. The first one considers a zero-measure perturbation of all $k$-dimensional subspaces contained in an $(n-1)$-subspace. The second one is the Cartesian product $K\times \R^{n-2}$ for a zero-measure planar Kakeya set $K$.  We will give the details in Section \ref{examples}.

When $\mu=\gamma_{n,k}$, using the embedding estimate 
\begin{align*}
    L^{n/k+O(\epsilon)}(\R^n)\hookrightarrow H^{n/2-k+\epsilon}(\R^n),
\end{align*} we have the following $L^p-L^2$ estimate for all $p>n/k$.
\begin{thm}\label{L^2-to-L^p}
    Suppose $2k>n$. Let $f$ be a measurable function with $\supp f\subset B(0,1)$. We have 
    \begin{align}
        \|\mathcal{M}^{n,k}(f)\|_{L^2(G(n,k))}\lesssim_{n,k,p} \|f\|_{L^p(\R^n)},\;\forall p>n/k.
    \end{align}
\end{thm}
As $G(n,k)$ is compact, our estimate is stronger than Falconer's $L^p-L^1$ estimate in \cite{Fal80}. Also, since $2>n/k$, by H\"older's inequality and interpolation with the trivial $L^{\infty}-L^{\infty}$ estimate, we obtain an $L^p-L^p$-estimate for all $p>n/k$. This is sharp by standard examples.
\subsection{$2k\leq n$}
Recall that in the range $2k>n$, we can reduce the $L^2$ estimate of the maximal operator $\mathcal{M}^{n,k}_{\delta}$ to an estimate of the operator 
\begin{align*}
    C_{\delta}(f)(V)=\delta^{k-n}\int_{\mathcal{N}_{\delta}(V)}f.
\end{align*}
In other words, we can ignore the effect of translations. However, when $2k\leq n$, the centered operator $C_\delta$
itself fails to satisfy a uniform $L^2$ estimate.
Moreover, for $p>2$, the translation-removal
principle does not hold in general, as demonstrated by
Proposition \ref{translation-not-remove}. Thus, the case $2k\leq n$ is quite different from the case $2k>n$, and we believe this explains why there are many proofs of positive measure for $(n,k)$-sets in the range $2k>n$, while there are few results in the range $2k\leq n$. 

Nevertheless, in the range $2k\leq n$, we can still obtain the following result for $\mathcal{M}^{n,k}$ with the $\epsilon$-loss Kakeya maximal estimate. The idea follows the lifting argument of Bourgain \cite{Bou91} and Oberlin \cite{Obe05}.
\begin{thm}\label{Kakeya-imply-(n,k,t)}
    Suppose $m,n,k$ are positive integers with $n-k=m$ and $k\geq 2$. Let $t>k(n-k-1)+1$. Suppose that for some $0<\epsilon<t-k(n-k-1)-1$ and some $2\leq p<\infty$, we have 
    \begin{align}
        \|\mathcal{M}^{m+1,1}_{\delta}(g)\|_{L^p(G(m+1,1))}\lesssim\delta^{-\epsilon/p}\|g\|_{L^p(\R^{m+1})},\;\supp g\subset B(0,1).
    \end{align}
    Then  we have 
    \begin{align}
        \|\mathcal{M}^{n,k}(f)\|_{L^p(\mu)}\lesssim\|f\|_{L^p(\R^n)},\;\supp f\subset B(0,1).
    \end{align}
    where $\mu$ is a $t$-dimensional Frostman measure on $G(n,k)$.
\end{thm}

Recently, Wang and Zahl \cite{WZ25} gave a quantitative Kakeya estimate in $\R^3$ and solved the Kakeya conjecture in $\R^3$. 
The quantitative Kakeya estimate in $\R^{m+1}$ says that for any $\epsilon>0$, there exists $K=K(\epsilon)>1$ such that the following holds for all small enough $\delta>0$. Let $\mathbb{T}$ be a set of $\delta$-tubes contained in the unit ball in $\R^{m+1}$ that point in $\delta$-separated directions. For each $T\in\mathbb{T}$, let $Y(T)\subset T$ be a measurable set with $|Y(T)|\geq \lambda|T|$. Then 
\begin{align*}
    \left|\bigcup _{T\in\mathbb{T}}Y(T)\right|\geq \delta^{\epsilon}\lambda^K\sum_{T\in\mathbb{T}}|T|.
\end{align*}

A simple corollary of this quantitative Kakeya estimate is the following $\epsilon$-loss Kakeya maximal estimate: for $\epsilon>0$, there exists $K=K(\epsilon)$ such that for $p>K(\epsilon)$,
\begin{align*}
    \|\mathcal{M}^{m+1,1}_{\delta}(g)\|_{L^p(G(m+1,1))}\lesssim\delta^{-\epsilon/p}\|g\|_{L^p(\R^{m+1})}.
\end{align*}

Combining this with Theorem \ref{Kakeya-imply-(n,k,t)}, we have the following.
    \begin{thm}\label{quantitative-Kakeya-imply-sharp-(n,k,t)}
        Assume that the quantitative Kakeya estimate in $\R^{m+1}$ holds.
        Then for $k\geq 2$ and $(n,k)=(k+m,k)$, $(n,k,t)$-sets have positive measure if $t>k(n-k-1)+1$.
    \end{thm}
Since the quantitative Kakeya estimate in $\R^{3}$ holds, we have the following sharp result when $n=4$, $k=2$.
    \begin{corollary}
        $(4,2,t)$-sets have positive measure if $t>3$.
    \end{corollary}
    
Theorem \ref{Kakeya-imply-(n,k,t)} is also valid in the range $2k>n$, but weaker than Theorem \ref{maximal-estimate-delta-free}; even if we have the sharp Kakeya maximal estimate, we can only get an $L^{n-k+1}$-estimate for the operator $\mathcal{M}^{n,k}$, while Theorem \ref{maximal-estimate-delta-free} implies an $L^2$ estimate. To obtain $L^p$-estimates for $p>n/k$, we need some methods other than the lifting argument.

\subsection*{Organization}
This paper is organized as follows. We first deal with maximal estimates, then show maximal estimates imply positive Lebesgue measure, and finally construct measure-zero examples showing the sharpness.
In Section \ref{maximal-function-2k>n}, we prove the maximal estimate, Theorem \ref{maximal-estimate-delta-free}, in the range
$2k>n$. The key reduction eliminates the supremum over translations: a $TT^*$
argument shows that the $L^2$ bound for the maximal operator is
controlled by that for a translation-free operator $C$. We also record an obstruction for this reduction, Proposition \ref{translation-not-remove}, when $2k\leq n$. In Section \ref{maximal-function-2k-leq-n}, we prove Theorem \ref{Kakeya-imply-(n,k,t)} via a lifting argument. In Section \ref{restricted-(n,k)-section}, we show that 
these estimates imply positive Lebesgue measure and
construct measure-zero examples showing the sharpness of the threshold $t>k(n-k-1)+1$.
\subsection*{Acknowledgement}
The author would like to thank his supervisor, Bochen Liu, for many useful suggestions. The author would also like to thank Shengwen Gan for pointing out several errors in the earlier version. The author began working on $(n,k)$-sets while attending \textit{the Summer School in p-adic Harmonic Analysis} at the Chern Institute of Mathematics, Nankai University, and would also like to thank the organizers, Shaoming Guo, Ping Xi, and Yakun Xi. 
\subsection*{Declaration of generative AI and AI-assisted technologies.}
Early in this work, the author used GPT-6 Astra as a brainstorming tool. The model produced the initial statement and a first proof of Theorem \ref{L^2}, together with an initial version of Proposition \ref{translation-not-remove} and its proof. The author then independently produced the stronger quantitative estimate, Theorem \ref{maximal-estimate-delta-free}, and its proof. For Theorem~\ref{quantitative-Kakeya-imply-sharp-(n,k,t)}, the author was familiar with Oberlin's lifting argument; during iterative discussions with the model, the model suggested that this argument extends to the Frostman measure setting and that one can use the quantitative Kakeya estimate on this problem.  The argument that appears in the paper, namely the proof of Theorem~\ref{Kakeya-imply-(n,k,t)} from which Theorem~\ref{quantitative-Kakeya-imply-sharp-(n,k,t)} follows, was developed by the author through these iterative discussions and differs substantially from the model's initial output.

All other results, arguments, and constructions in this paper are the author's own work, and the author wrote the entire manuscript. The author used Doubao 2.1 Pro, DeepSeek V4.1 Flash, and GPT-6.2 Sol to assist with checking, revising, and polishing the exposition. The author reviewed and verified all AI-generated outputs and accepts full responsibility for the originality, correctness, and integrity of this manuscript.

\section{Maximal Operators, $2k>n$}\label{maximal-function-2k>n}
In this section, we prove Theorem \ref{maximal-estimate-delta-free}, explain the sharpness of Theorem \ref{maximal-estimate-delta-free} and the limitations of our method when $2k\leq n$.

\subsection{From $\mathcal{M}^{n,k}_{\delta}$ to $\mathcal{M}^{n,k}$}
First, we show that it suffices to show Theorem \ref{maximal-estimate-delta-free} with $\mathcal{M}^{n,k}$ replaced by $\mathcal{M}^{n,k}_{\delta}$. More precisely, we have the following.
\begin{lem}\label{delta-to-delta-free}
    Suppose $f\geq 0$. For $s\leq 0$, the uniform estimate 
\begin{align*}
    \|\mathcal{M}^{n,k}_{\delta}(f)\|_{L^2(\mu)}\lesssim \|f\|_{H^s(\R^n)},\;\forall\delta>0
\end{align*}
implies
\begin{align*}
    \|\mathcal{M}^{n,k}(f)\|_{L^2(\mu)}\lesssim \|f\|_{H^s(\R^n)}.
\end{align*}
And the converse is trivial, so indeed the two estimates are equivalent up to constants.
\end{lem}
\begin{proof}
    The proof is a standard argument.
First we consider continuous functions with compact support.
Suppose $f\in C(\R^n)$ with $\supp f\subset B(0,1)$.   We choose $\phi\in C_0^{\infty}(B(0,1))$ with $\phi\geq 0$,  $\int\phi=1$. For $\delta>0$, we define $\phi_{\delta}(x)=\delta^{-n}\phi(x/\delta)$. Then 
\begin{align*}
    \int_{x+V}(f*\phi_{\delta})(u)d\H^k(u)\lesssim \delta^{-n}\int_{V}\int f(x+u-y)\phi(y/\delta)dyd\H^k(u).
\end{align*}
We consider the decomposition $y=a+b$, $a\in V$, $b\in V^{\bot}$. Then by Fubini's theorem,
\begin{align*}
    &\int_{V}\int f(x+u-y)\phi(y/\delta)dyd\H^k(u)\\
    =&\int_V\int_{V^{\bot}}\left(\int_Vf(x+u-a-b)d\H^k(u)\right)\phi((a+b)/\delta)d\H^k(a)d\H^{n-k}(b)\\
     =&\int_V\int_{V^{\bot}}\left(\int_Vf(x+u-b)d\H^k(u)\right)\phi((a+b)/\delta)d\H^k(a)d\H^{n-k}(b).
\end{align*}
Also, 
\begin{align*}
    \int_V\phi((a+b)/\delta)d\H^k(a)\lesssim \delta^k \cdot1_{B(0,\delta)\cap V^{\bot}}(b).
\end{align*}
Then 
\begin{align*}
    &\int_{V}\int f(x+u-y)\phi(y/\delta)dyd\H^k(u)\\
    \lesssim&\delta^k\cdot\int_{B(0,\delta)\cap V^{\bot}}\int_Vf(x+u-b)d\H^k(u)d\H^{n-k}(b)\\
    \lesssim& \delta^{k}\cdot\int_{\mathcal{N}_{\delta}(x+V)}f.
\end{align*}
Hence 
\begin{align*}
    \mathcal{M}^{n,k}(f*\phi_{\delta})(V)\lesssim \mathcal{M}_{\delta}^{n,k}(f)(V),\;V\in G(n,k).
\end{align*}
Then we have 
\begin{align*}
    \|\mathcal{M}^{n,k}(f*\phi_{\delta})\|_{L^2(\mu)}\lesssim \|\mathcal{M}_{\delta}^{n,k}(f)\|_{L^2(\mu)}\lesssim\|f\|_{H^s(\R^n)},\;\forall\delta>0.
\end{align*}
On the other hand, since $f$ is continuous with compact support, $f*\phi_{\delta}$ converges uniformly to $f$ as $\delta\to 0$. Then 
\begin{align*}
    \mathcal{M}^{n,k}(f)(V)\leq \lim_{\delta\to0}\inf\mathcal{M}^{n,k}(f*\phi_{\delta})(V).
\end{align*}
By Fatou's lemma, we have 
\begin{align*}
    \|\mathcal{M}^{n,k}(f)\|_{L^2(\mu)}\leq\lim_{\delta\to0}\inf\|\mathcal{M}^{n,k}(f*\phi_{\delta})\|_{L^2(\mu)}\lesssim\|f\|_{H^s(\R^n)}.
\end{align*}

Then we turn to the general case. Suppose $f$ is a non-negative, Lebesgue measurable function with $\supp f\subset B(0,1)$. For any $\epsilon>0$, we can choose a lower semicontinuous function $g$ such that $f\leq g$ and $\|f-g\|_{L^2(\R^n)}<\epsilon$. Let $g_\delta=g*\phi_{\delta}\in C(\R^n)$,  $\supp g*\phi_{\delta}\subset B(0,2)$.
By the previous result on continuous functions with compact support, we have
\begin{align*}
   \|\mathcal{M}^{n,k}(g_\delta)\|_{L^2(\mu)}\lesssim\|g_\delta\|_{H^s(\R^n)}\leq\|f\|_{H^s(\R^n)}+\epsilon.
\end{align*}
On the other hand, since $g$ is lower semicontinuous, $g\leq\lim\limits _{\delta\to0}\inf g_{\delta}$. By Fatou's lemma,
\begin{align*}
    \|\mathcal{M}^{n,k}(f)\|_{L^2(\mu)}\leq\|\mathcal{M}^{n,k}(g)\|_{L^2(\mu)}\leq \lim\limits _{\delta\to0}\inf\|\mathcal{M}^{n,k}(g_\delta)\|_{L^2(\mu)}\lesssim\|f\|_{H^s(\R^n)}+\epsilon.
\end{align*}
Let $\epsilon\to0$, and we complete the proof.
\end{proof}

\subsection{$L^2$-estimates for $\mathcal{M}^{n,k}_{\delta}$}
First, we can drop the supremum and just consider the $k$-plane transform.  This is a standard measurable selection argument.
\begin{lem}\label{Borel-selection}
    For any measurable $f$ with $\supp f\subset B(0,1)$, we can find a Borel measurable map $V\mapsto x_V\in \R^n$ such that 
\begin{align}
    \mathcal{M}_{\delta}^{n,k}(f)(V)\lesssim \delta^{k-n}\int_{\mathcal{N}_{\delta}(x_V+V)}|f|,\;\forall V\in G(n,k).
\end{align}
\end{lem}
\begin{proof}
    We can assume that $f$ is nonnegative and Borel measurable. For each $V$ and $x$, we define 
    \begin{align*}
        F(V,x)=\delta^{k-n}\int_{\mathcal{N}_{\delta}(x+V)}f.
    \end{align*}
    For fixed $x$, the function $V\mapsto F(V,x)$ is Borel measurable. 
    We choose a countable dense subset $\{q_j\}_{j\geq1}\subset \R^n$. If $f\in L^1$, the function $x\mapsto F(V,x)$ is continuous for fixed $V$.  Since $f$ has bounded support, by continuity, for any $V$, $\mathcal{M}_{\delta}^{n,k}(f)(V)=\sup_{j\geq1}F(V,q_j)$.
    If $f\notin L^1$, since $B(0,1)$ is precompact, there exists $y\in B(0,1)$ such that $f$ is not integrable over any neighborhood of $y$. In particular, $F(V,y)=\infty$. Choose $q_j$ such that $|y-q_j|<\delta$, then $\mathcal{N}_{\delta}(q_j+V)$ contains a neighborhood of $y$ for all $V$. Then $\mathcal{M}_{\delta}^{n,k}(f)(V)=\sup_{j\geq1}F(V,q_j)=\infty$. In both cases, we have 
    \begin{align*}
        \mathcal{M}_{\delta}^{n,k}(f)(V)=\sup_{j\geq1}F(V,q_j).
    \end{align*}
    
    For $j\geq 1$, we define
    \begin{align*}
        E_j=\{V\in G(n,k):F(V,q_j)\geq \frac{1}{2}\mathcal{M}_{\delta}^{n,k}(f)(V)\}
    \end{align*}
    $E_j$ is Borel since $\mathcal{M}_{\delta}^{n,k}(f)$ and $V\mapsto F(V,x)$ are Borel measurable. And $\{E_j\}_{j\geq1}$ covers $G(n,k)$. Then we define 
    \begin{align*}
        j(V)=\min\{j\geq 1:V\in E_j\}
    \end{align*}
    and 
    \begin{align*}
        x_V=q_{j(V)}=\sum_{j\geq 1}q_j1_{\{j(V)=j\}}.
    \end{align*}
    By construction of $E_j$, $V\mapsto x_V$ is Borel measurable and 
    \begin{align*}
        \mathcal{M}_{\delta}^{n,k}(f)(V)\leq 2\delta^{k-n}\int_{\mathcal{N}_{\delta}(x_V+V)}f,\;\forall V\in G(n,k).
    \end{align*}
\end{proof}

Then, if we focus on the \(L^2\)-norm, we can reduce $P_{\delta}$ to a translation-free operator, $C$. In other words, we can ignore the effect of the translation $V\mapsto x_V$.
\begin{lem}\label{translation-remove}
    Let
\begin{align*}
    P_{\delta}(f)(V)=\delta^{k-n}\int_{\mathcal{N}_{\delta}(x_V+V)}f,
\end{align*}
where $V\mapsto x_V$ is given by Lemma \ref{Borel-selection},
and let
$C$ be the operator given by
\begin{align*}
    C(f)(V)=\int_{V}fd\H^k.
\end{align*}
Suppose that for some $s\in\R$, we have  
\begin{align}\label{Assmption-H^s-to-L^2}
    \|C(f)\|_{L^2(\mu)}\lesssim \|f\|_{H^s(\R^n)},\;\supp f\subset B(0,1).
\end{align}
Then for the same $s$, we have 
\begin{align}
 \|P_{\delta}(f)\|_{L^2(\mu)}\lesssim \||f|\|_{H^s(\R^n)},\;\supp f\subset B(0,1).
\end{align}
\end{lem}
\begin{proof}
    The idea is to replace the rough tube by a smooth tube with a nonnegative Fourier transform. Translating its center only multiplies by a phase of modulus 1 in the Fourier side, so the associated operator norm cannot increase.  

    We choose two nonnegative radial functions $\phi,\chi\in C^{\infty}_0(B(0,4))$ such that $\phi,\chi\geq 1 $ on $B(0,3)$ and $\hat{\phi},\hat{\chi}\geq 0$. Let $\phi_{\delta}(x)=\delta^{-n}\phi(x/\delta)$. We define 
    \begin{align*}
        H_{V,\delta}(x)=\int_{V}\phi_{\delta}(x-v)d\H^kv
    \end{align*}
    and 
    \begin{align*}
        \Phi_{V,\delta}(x)=\chi(x)\cdot H_{V,\delta}(x).
    \end{align*}
    We consider the operators
    \begin{align*}
        Q_{\delta}(f)(V)=\int_{\R^n}\Phi_{V,\delta}(x-x_V)f(x)dx
    \end{align*}
    and 
    \begin{align*}
        D_{\delta}(f)(V)=\int_{\R^n}\Phi_{V,\delta}(x)f(x)dx.
    \end{align*}
    These two operators are smooth versions of $P_{\delta}$, $C_{\delta}$ respectively.

    Note that $P_{\delta}$ is dominated by $Q_{\delta}$. We can assume $|x_V|\leq 2$ since $\supp f\subset B(0,1)$. Then $\chi(x-x_V)\geq 1_{B(0,1)}$.
    And 
    \begin{align*}
        H_{V,\delta}(x-x_V)&=\int_{V}\phi_{\delta}(x-x_V-v)d\H^kv\\
        &=\delta^{k-n}\int_{V}\phi((x-x_V)/\delta-v)d\H^kv.
    \end{align*}
    When $x\in \mathcal{N}_{\delta}(x_V+V)$, $d((x-x_V)/\delta, V)\leq 1$. We have 
    \begin{align*}
        H_{V,\delta}(x-x_V)\gtrsim_{n,k}\delta^{k-n}\cdot1_{\mathcal{N}_{\delta}(x_V+V)}
    \end{align*}
    Hence we have 
    \begin{align}
        \delta^{k-n}\cdot1_{B(0,1)\cap\mathcal{N}_{\delta}(x_V+V)}(x)\leq \Phi_{V,\delta}(x-x_V).
    \end{align}
    As a result,
    \begin{align}
        \|P_{\delta}(f)\|_{L^2(\mu)}\lesssim_{n,k}\|Q_{\delta}(|f|)\|_{L^2(\mu)}.
    \end{align}
    
    Also, $D_{\delta}$ is dominated by $C$.
    By a direct computation,
    \begin{align*}
        D_{\delta}(f)=C(\phi_{\delta} *\chi f).
    \end{align*}
    Note that $\supp \phi_{\delta} *\chi f\subset B(0,8)$, by the assumption (\ref{Assmption-H^s-to-L^2}), we have 
    \begin{align}
       \| D_{\delta}(f)\|_{L^2(\mu)}\lesssim\|\phi_{\delta} *\chi f\|_{H^s(\R^n)}\lesssim \|f\|_{H^s(\R^n)}.
    \end{align}
    
    Now it suffices to show that 
    \begin{align}
        \|Q_{\delta}\|_{H^s(\R^n)\mapsto L^2(\mu)}\lesssim \|D_{\delta}\|_{H^s(\R^n)\mapsto L^2(\mu)}.
    \end{align}
    We prove it using the Fourier transform and the $TT^*$ method.  By a direct computation, 
     \begin{align*}
        Q_{\delta}Q_{\delta}^*G(V)=\int K_{Q_{\delta}}(V,W)G(W)d\mu(W),\;G\in L^2(\mu),
    \end{align*}
    where
    \begin{align}
        K_{Q}(V,W)=\int_{\R^n}e^{-2\pi i(x_V-x_W)\cdot\xi}\widehat{\Phi_{V,\delta}}(\xi)\widehat{\Phi_{W,\delta}}(\xi)(1+|\xi|^2)^{-s}d\xi.
    \end{align}
    Similarly, the kernel of $D_{\delta}D_{\delta}^*$ is given by
    \begin{align}
        K_{D}(V,W)=\int_{\R^n}\widehat{\Phi_{V,\delta}}(\xi)\widehat{\Phi_{W,\delta}}(\xi)(1+|\xi|^2)^{-s}d\xi.
    \end{align}
    Note that
    \begin{align*}
        \widehat{\Phi_{V,\delta}}(\xi)=(\hat{\chi}*\widehat{H_{V,\delta}})(\xi)=\int_{V^{\bot}}\hat{\phi}(\delta\eta)\hat{\chi}(\xi-\eta)d\H^{n-k}(\eta)\geq 0,
    \end{align*} 
    we have 
    \begin{align}
        |K_{Q}(V,W)|\leq K_{D}(V,W),\;\forall V,W.
    \end{align}
    Then 
    \begin{align*}
        \|Q_{\delta}^*G\|_{H^{s}(\R^n)}\leq \|D_{\delta}^*|G|\|_{H^{s}(\R^n)}\lesssim \|D_{\delta}\|_{H^s(\R^n)\mapsto L^2(\mu)}\cdot\|G\|_{L^2(\mu)}.
    \end{align*}
    Thus 
    \begin{align*}
        \|Q_{\delta}\|_{H^s(\R^n)\mapsto L^2(\mu)}\lesssim \|D_{\delta}\|_{H^s(\R^n)\mapsto L^2(\mu)}.
    \end{align*}
\end{proof}

\begin{remark}
    There is a simple geometric proof for Lemma \ref{translation-remove} at $s=0$. We can take $\phi\in C^{\infty}_0(\R^n)$ with $\phi\geq 0$ and $\phi\gtrsim 1$ on $B(0,2)$. Let $\phi_{\delta}(x)=\delta^{-n}\phi(x/\delta)$. Since $\supp f\subset B(0,1)$, we have 
\begin{align*}
    C_{\delta}(f)(V)\lesssim C(|f|*\phi_{\delta})(V),\;\forall\delta>0,
\end{align*}
where $C_{\delta}$ is defined by 
\begin{align*}
    C_{\delta}(f)(V)=\delta^{k-n}\int_{\mathcal{N}_{\delta}(V)}f.
\end{align*}
Then 
\begin{align*}
    \|C_{\delta}\|_{L^2(B(0,1))\to L^2(\mu)}\lesssim \|C\|_{L^2(B(0,1))\to L^2(\mu)}
\end{align*}
    Let $K_P$, $K_C$ denote the kernels of $P_{\delta}P^*_{\delta}$ and $C_{\delta}C^*_{\delta}$. By a direct computation, we have 
    \begin{align*}
        K_P(V,W)&=\delta^{2k-2n}|\mathcal{N}_{\delta}(x_V+V)\cap \mathcal{N}_{\delta}(x_W+W)\cap B(0,1)|,\\
        K_C(V,W)&=\delta^{2k-2n}|\mathcal{N}_{\delta}(V)\cap \mathcal{N}_{\delta}(W)\cap B(0,1)|.
    \end{align*}
    The sets $B(0,1)$, $\mathcal{N}_{\delta}(V)$, $\mathcal{N}_{\delta}(W)$ are all convex and centrally symmetric. Let 
    \begin{align*}
        E=\mathcal{N}_{\delta}(x_V+V)\cap \mathcal{N}_{\delta}(x_W+W)\cap B(0,1).
    \end{align*}
    We have 
    \begin{align*}
        \frac{E-E}{2}\subset\mathcal{N}_{\delta}(V)\cap \mathcal{N}_{\delta}(W)\cap B(0,1)
    \end{align*}
    By the Brunn-Minkowski inequality,
    \begin{align*}
        K_P(V,W)\leq K_C(V,W).
    \end{align*}
    Hence 
    \begin{align*}
        \|P_{\delta}\|_{L^2(B(0,1))\to L^2(\mu)}\leq \|C_{\delta}\|_{L^2(B(0,1))\to L^2(\mu)}.
    \end{align*}
\end{remark}

Thus it suffices to show that 
\begin{align}\label{L^2-to-Sobolev}
    \|C(f)\|_{L^2(\mu)}\lesssim_{\epsilon}
    \begin{cases}
        \|f\|_{H^{(k(n-k-1)+1-t)/2+\epsilon}(\R^n)},\;t\leq(k-1)(n-k)+1;\\
        \|f\|_{H^{n/2-k+\epsilon}(\R^{n})},\;t\geq(k-1)(n-k)+1.
    \end{cases}.
\end{align}

Let
\begin{align*}
    R(f)(\omega)=\int_0^1f(r\omega)r^{k-1}dr,\;\omega\in S^{n-1},
\end{align*}
and let $S$ denote the spherical average given by
\begin{align*}
    S(F)(V)=\int_{V\cap S^{n-1}}F(\omega)d\sigma_V^{k-1}(\omega),
\end{align*}
where $\sigma_V^{k-1}$ is the standard surface measure on the sphere $V\cap S^{n-1}$.
By the polar coordinate formula, we can rewrite $C(f)$ as
\begin{align*}
    C(f)(V)=\int_{V\cap S^{n-1}}\int_0^1f(r\omega)r^{k-1}drd\sigma_V^{k-1}(\omega)=SR(f)(V).
\end{align*}

We will estimate the operators $R$ and $S$ respectively.
For the operator $R$, we have the following Sobolev norm estimate.
\begin{lem}\label{Radial-operator} Suppose $\supp f\subset B(0,1)$. We have 
    \begin{align}
    \|R(f)\|_{H^\alpha(S^{n-1})}\lesssim\|f\|_{H^{\beta}(\R^n)},\;\beta\geq \alpha,\;\alpha>n/2-k.
\end{align}
\end{lem}
\begin{proof}
    We choose a cut-off function $\chi\in C^{\infty}_0(B(0,2))$ such that $\chi=1$ on $B(0,1)$.
    By definition, we have 
    \begin{align*}
        \langle R(f),G\rangle&=\int_{S^{n-1}}\int_0^1\chi(r\omega)f(r\omega)r^{k-1}drG(\omega)d\sigma^{n-1}(\omega)\\
        &=\int_0^1r^{k-1}\int_{S^{n-1}}\chi(r\omega)f(r\omega)G(\omega)d\sigma^{n-1}(\omega)dr.
    \end{align*}
    By the polar coordinate formula,
    \begin{align*}
        \int_0^1r^{k-1}\int_{S^{n-1}}\chi(r\omega)f(r\omega)G(\omega)d\sigma^{n-1}(\omega)dr=\int_{\R^n}\chi(x)f(x)|x|^{k-n}G(x/|x|)dx.
    \end{align*}
    Thus, the dual operator $R^*$ is a weighted radial extension operator:
    \begin{align}
        R^*(G)(x)=\chi(x)|x|^{k-n}G(x/|x|),\;x\in\R^n.
    \end{align}
    It suffices to show 
    \begin{align}
        \|R^*(G)\|_{H^s(\R^n)}\lesssim\|G\|_{H^s(S^{n-1})},\;s<k-n/2.
    \end{align}
    For related estimates for radial extension operators, see
    \cite[Lemma 3]{BMS19}. For completeness, we give a short
    proof based on a dyadic decomposition.
    
    The origin is the only singularity. We consider a dyadic decomposition around the origin.
    Let $\psi(x)=\chi(x)-\chi(2x)$ and 
    \begin{align*}
    \chi(x)=\sum_{j=0}^{\infty}\chi(2^jx)-\chi(2^{j+1}(x))=\sum_{j=0}^{\infty}\psi(2^jx),\;x\neq0.
    \end{align*}
    Then we rewrite
    \begin{align*}
        R^*(G)(x)=\sum_{j=0}^{\infty}\psi(2^jx)|x|^{k-n}G(x/|x|).
    \end{align*}
    Let 
    \begin{align*}
        F(x)=\psi(x)|x|^{k-n}G(x/|x|).
    \end{align*}
    By homogeneity,
    \begin{align*}
        R^*(G)(x)=\sum_{j=0}^{\infty}\psi(2^jx)|x|^{k-n}G(x/|x|)=\sum_{j=0}^{\infty}2^{j(n-k)}F(2^jx).
    \end{align*}
    Then by definition,
    \begin{align*}
        \|F(2^j\cdot)\|^2_{H^s(\R^n)}=2^{-2jn}\int(1+|\xi|^2)^s|\hat{F}(2^{-j}\xi)|^2d\xi\lesssim 2^{-jn+2j\max\{s,0\}}\|F\|^2_{H^s(\R^n)}.
    \end{align*}
    Since $s<k-n/2$ and $2k>n$, we have 
    \begin{align*}
        \|R^*(G)\|_{H^s(\R^n)}\leq \sum_{j=0}^{\infty}2^{j(n/2-k+\max\{s,0\})}\|F\|_{H^s(\R^n)}\lesssim \|F\|_{H^s(\R^n)}.
    \end{align*}
    Finally, note that $\psi$ is supported in the annulus $\{1/2\leq|x|\leq 2\}$ away from the origin, we have 
    \begin{align*}
        \|F\|_{H^s(\R^n)}\lesssim_s\|G\|_{H^s(S^{n-1})}.
    \end{align*}
    \end{proof}
For the operator $S$, we have the following.
\begin{lem}\label{Spherical-operator}
    For any $\epsilon>0$, we have
    \begin{align}\label{spherical-estimate}
    \|S(F)\|_{L^2(\mu)}\lesssim_{\epsilon}\|F\|_{H^{(k(n-k-1)+1-t)/2+\epsilon}(S^{n-1})}.
    \end{align}
\end{lem}
\begin{proof}
Roughly speaking, the estimate (\ref{spherical-average-s-gain}) says the spherical average operator increases the regularity by $(k-1)/2$, while the estimate (\ref{integration-fractal-s-loss}) says integration over $\mu$ decreases the regularity by $\frac{1}{2}(k(n-k)-t)$.

Note that both $S^{n-1}$ and $G(n,k)$ are $SO(n)$-homogeneous spaces, and 
the operator $S$ is $SO(n)$-equivariant: for any orthogonal map $T\in SO(n)$, we have  
\begin{align*}
    S(G\circ T)=S(G)\circ T.
\end{align*}
With compatible standard metric normalizations,  $S$ also transfers $\Delta_{S^{n-1}}$ on $S^{n-1}$ to $\Delta_{G(n,k)}$ on $G(n,k)$:
\begin{align*}
    S(\Delta_{S^{n-1}}G)=\Delta_{G(n,k)}S(G).
\end{align*}
We refer to \cite[Chapter II, \S2]{Hel99} for such a commutation argument.
Note that both $\Delta_{S^{n-1}}$ and $\Delta_{G(n,k)}$ are self-adjoint operators, a standard functional calculus argument gives
\begin{align}\label{Laplacian-Equivalent}
    S((I-\Delta_{S^{n-1}})^sG)=(I-\Delta_{G(n,k)})^sS(G),\;\forall s\in \R.
\end{align}
Thus we  can decompose the 
estimate (\ref{spherical-estimate}) into the following two estimates: 
\begin{align}\label{spherical-average-s-gain}
    \|S(G)\|_{H^{(k-1)/2}(G(n,k))}\lesssim \|G\|_{L^2(S^{n-1})}.
\end{align}
\begin{align}\label{integration-fractal-s-loss}
    \|g\|_{L^2(\mu)}\lesssim_{\epsilon} \|g\|_{H^{(k(n-k)-t)/2+\epsilon}(G(n,k))},\;\forall \epsilon>0.
\end{align}

We first deal with (\ref{spherical-average-s-gain}). The operator $S$ is well understood: by the Funk–Hecke type theorem for great-subsphere averages developed by \cite{Rub24}, $S$ maps each spherical harmonic to a Grassmannian harmonic, with scalar decay of order $-(k-1)/2$. Thus (\ref{spherical-average-s-gain}) holds.
For completeness, we give a sketch of a proof. 
By (\ref{Laplacian-Equivalent}), we can write  
\begin{align*}
    \|S(G)\|_{H^s(G(n,k))}^2=\langle (I-\Delta_{S^{n-1}})^{s/2}S^*S(I-\Delta_{S^{n-1}})^{s/2}G, G\rangle_{S^{n-1}}.
\end{align*}
It suffices to consider the operator $S^*S$ and show that 
\begin{align}
    \|S^*S(G)\|_{H^{k-1}(S^{n-1})}\lesssim\|G\|_{L^2(S^{n-1})}
\end{align} We first compute $S^*$ and $S^*S$. We consider the fiber bundle
\begin{align*}
    M=\{(V,\omega)\in G(n,k)\times S^{n-1}:\omega\in V\cap S^{n-1}\}
\end{align*}
and the measure $\nu$ on $M$ given by
\begin{align}
    d\nu(V,\omega)=d\sigma_V^{k-1}(\omega)d\gamma_{n,k}(V).
\end{align}
Note that for any orthogonal map $T\in O(n)$,
\begin{align*}
    d\nu(V,\omega)=d\nu(TV,T\omega).
\end{align*}
Let $\pi_{S^{n-1}}$ denote the projection from $M$ to $S^{n-1}$, then the pushforward measure $(\pi_{S^{n-1}})_*\nu$ is a rotation-invariant measure on $S^{n-1}$. So
\begin{align*}
    (\pi_{S^{n-1}})_*\nu=c_{n,k}\cdot\sigma^{n-1}
\end{align*}
for some constant $c_{n,k}$ depending on $n,k$. We can rewrite $d\nu$ by
\begin{align}
    d\nu(V,\omega)=c_{n,k}\cdot d\gamma_{\omega}(V)d\sigma^{n-1}(\omega),
\end{align}
where $\gamma_{\omega}$ is the invariant measure on $\{V\in G(n,k):\omega\in V\}$. Then for $G\in L^2(S^{n-1})$, $g\in H^s$,
\begin{align*}
    \langle S(G),g\rangle=&\int_{G(n,k)} S(G)(V)g(V)d\gamma_{n,k}\\
    =&c_{n,k}\int_{S^{n-1}}G(\omega)\int_{\{V\in G(n,k):\;\omega\in V\}}g(V)d\gamma_{\omega}(V)d\sigma^{n-1}(\omega).
\end{align*}
Thus 
\begin{align}
    S^*(g)(\omega)=c_{n,k}\int_{\{V\in G(n,k):\;\omega\in V\}}g(V)d\gamma_{\omega}(V)
\end{align}
and 
\begin{align}
    S^*S(G)(\omega)=c_{n,k}\int_{\{V\in G(n,k):\;\omega\in V\}}\int_{V\cap S^{n-1}}G(\eta)d\sigma_V^{k-1}(\eta)d\gamma_{\omega}(V).
\end{align}

Then we compute the kernel of $S^*S$. For $\eta\in V\cap S^{n-1}$ with $\omega\in V$, we can write
\begin{align*}
    \eta=u\omega+\sqrt{1-u^2}\theta,\;u=\langle \omega,\eta\rangle\in[-1,1],\; \theta \in \omega^{\bot}\cap S^{n-1}.
\end{align*}
By the formula for the change of coordinates,
\begin{align*}
    &\int_{V\cap S^{n-1}}G(\eta)d\sigma_V^{k-1}(\eta)\\
    =&\int_{-1}^1(1-u^2)^{(k-3)/2}\int_{\omega^{\bot}\cap V\cap S^{n-1}}G(u\omega+\sqrt{1-u^2}\theta)\;d\sigma^{k-2}_{\omega^{\bot}\cap V }(\theta)du.
\end{align*}
For $V$ containing $\omega$, we can write $V=\Span{\omega}\oplus U$ with $U\in G(\omega^{\bot},k-1)$. The measure $\gamma_{\omega}$ is just the invariant probability measure on $G(\omega^{\bot},k-1)$. We have 
\begin{align*}
    &S^*S(G)(\omega)\\
    =&c_{n,k}\int_{-1}^1(1-u^2)^{(k-3)/2}\int_{\{V\in G(n,k):\;\omega\in V\}}\int_{\omega^{\bot}\cap V\cap S^{n-1}}G(u\omega+\sqrt{1-u^2}\theta)\;d\sigma^{k-2}_{\omega^{\bot}\cap V }(\theta)d\gamma_{\omega}(V)du\\
    =&c_{n,k}\int_{-1}^1(1-u^2)^{(k-3)/2}\int_{G(\omega^{\bot},k-1)}\int_{U\cap S^{n-1}}G(u\omega+\sqrt{1-u^2}\theta)d\sigma^{k-2}_{U }(\theta)d\gamma_{n-1,k-1}(U)du.
\end{align*}
Similarly to the measure $\nu$, we can change the order:
\begin{align*}
    d\sigma^{k-2}_{U }(\theta)d\gamma_{n-1,k-1}(U)=d\gamma_{\theta}(U)d\sigma^{n-2}_{\omega^{\bot}}(\theta),
\end{align*}
where $\gamma_{\theta}$ is the invariant measure on $\{U\in G(\omega^{\bot},k-1):\theta\in U\}$.
We have 
\begin{align*}
    &\int_{G(\omega^{\bot},k-1)}\int_{U\cap S^{n-1}}G(u\omega+\sqrt{1-u^2}\theta)d\sigma^{k-2}_{U }(\theta)d\gamma_{n-1,k-1}(U)\\
    =&\int_{\omega^{\bot}\cap S^{n-1}}G(u\omega+\sqrt{1-u^2}\theta)\int_{\{U\in G(\omega^{\bot},k-1\}}d\gamma_{\theta}(U)d\sigma^{n-2}_{\omega^{\bot}}(\theta)\\
    =&\int_{\omega^{\bot}\cap S^{n-1}}G(u\omega+\sqrt{1-u^2}\theta)d\sigma^{n-2}_{\omega^{\bot}}(\theta).
\end{align*}
Thus 
\begin{align*}
    S^*S(G)(\omega)=c_{n,k}\int_{-1}^1(1-u^2)^{(k-3)/2}\int_{\omega^{\bot}\cap S^{n-1}}G(u\omega+\sqrt{1-u^2}\theta)d\sigma^{n-2}_{\omega^{\bot}}(\theta)du.
\end{align*}
We consider the change of variables $\eta=u\omega+\sqrt{1-u^2}\theta$, and get 
\begin{align}
    S^*S(G)(\omega)=c_{n,k}\int_{S^{n-1}}(1-\langle \omega,\eta\rangle^2)^{(k-n)/2}G(\eta)d\sigma^{n-1}(\eta).
\end{align}
The kernel is given by 
\begin{align}
    K(\omega,\eta)=(1-\langle \omega,\eta\rangle^2)^{(k-n)/2}\sim|\omega+\eta|^{k-n}\cdot|\omega-\eta|^{k-n}.
\end{align}
In local coordinates, it is either smooth or roughly a Riesz kernel of order $k-1$. Thus, $S^*S$ is bounded from $L^2$ to $H^{k-1}$.

The estimate (\ref{integration-fractal-s-loss}) follows from the assumption that $\mu$ is a $t$-dimensional Frostman measure with $t>k(n-k-1)+1$. It suffices to work in a local coordinate. Let $z\in\R^{k(n-k)}$ denote the coordinate. And let $s=(k(n-k)-t)/2+\epsilon$.
By definition,
\begin{align*}
    \|g\|_{L^2(\mu)}^2\lesssim \|(I-\Delta)^{s/2}g\|_{L^2}\cdot \|(I-\Delta)^{-s/2}(gd\mu)\|_{L^2}\lesssim\|g\|_{H^s}\cdot\|(I-\Delta)^{-s/2}(gd\mu)\|_{L^2}.
\end{align*}
And by the Plancherel formula,
\begin{align*}
    \|(I-\Delta)^{-s/2}(gd\mu)\|_{L^2}^2&=\int_{\R^{k(n-k)}}(1+|\xi|^2)^{-s}|\widehat{gd\mu}(\xi)|^2d\xi\\
    &\leq\int_{|\xi|\geq 1}|\xi|^{-2s}|\widehat{gd\mu}(\xi)|^2d\xi+\int_{|\xi|\leq 1}|\widehat{gd\mu}(\xi)|^2d\xi.
\end{align*}
For the second integral, we use a trivial estimate,
\begin{align*}
    |\widehat{gd\mu}(\xi)|\leq \int |g|d\mu\lesssim \|g\|_{L^2(\mu)}.
\end{align*}
For the first one, we use the fact that up to constants, the Fourier inversion of $|\xi|^{-2s}$ is $|z|^{-t+2\epsilon}$:
\begin{align*}
    \int_{|\xi|\geq 1}|\xi|^{-2s}|\widehat{gd\mu}(\xi)|^2d\xi\lesssim 
    \int_{\R^{k(n-k)}}\int_{\R^{k(n-k)}}|z-w|^{-t+2\epsilon}g(z)g(w)d\mu(z)d\mu(w).
\end{align*}
Using the inequality $|ab|\lesssim |a|^2+|b|^2$, and the symmetry of $|z-w|^{-t+2\epsilon}$, we have 
\begin{align*}
    &\int_{\R^{k(n-k)}}\int_{\R^{k(n-k)}}|z-w|^{-t+2\epsilon}g(z)g(w)d\mu(z)d\mu(w)\\
    \lesssim& \int_{\R^{k(n-k)}}|g(w)|^2\left(\int_{\R^{k(n-k)}}|z-w|^{-t+2\epsilon}d\mu(z)\right)d\mu(w).
\end{align*}
Recall that $\mu$ is a $t$-dimensional Frostman measure with $t>k(n-k-1)+1$, a standard dyadic argument gives
\begin{align*}
    \sup_{w}\int_{\R^{k(n-k)}}|z-w|^{-t+2\epsilon}d\mu(z)\lesssim_{\epsilon} 1.
\end{align*}
Thus
\begin{align*}
    \|(I-\Delta)^{-s/2}(gd\mu)\|_{L^2}^2\lesssim_{\epsilon} \|g\|_{L^2(\mu)}^2
\end{align*}
and we complete the proof of (\ref{integration-fractal-s-loss}).
\end{proof}
\subsection{Sharpness of Theorem \ref{maximal-estimate-delta-free}}\label{sharpness-of-1.1} 
In this section, we show that Theorem \ref{maximal-estimate-delta-free} is sharp up to $\epsilon$-loss by constructing explicit examples. 

In the range $t\geq (k-1)(n-k)+1$, we choose a smooth function $\varphi\in C^{\infty}_0(B(0,1))$ such that $\varphi\geq 0$ and $\varphi =1$ on $B(0,1/2)$. We consider  
\begin{align*}
    f_{\delta}(x)=\varphi(x/\delta).
\end{align*}
By a direct computation, $\int_Vf_{\delta}(x)d\H^k(x)\gtrsim \delta^k$, then 
\begin{align*}
    \|\mathcal{M}^{n,k}(f_{\delta})\|_{L^2(\mu)}\gtrsim \delta^k.
\end{align*}
And when $-n/2<s\leq 0$,
\begin{align*}
\|f_{\delta}\|_{H^s(\R^n)}\lesssim\delta^{n/2-s}.
\end{align*}
Hence 
\begin{align*}
    \frac{\|\mathcal{M}^{n,k}(f_{\delta})\|_{L^2(\mu)}}{\|f_{\delta}\|_{H^s(\R^n)}}\geq \delta^{k-n/2+s}.
\end{align*}
When $s<n/2-k$, $\delta^{k-n/2+s}\to \infty$.

It is more complicated for the range $k(n-k-1)+1<t<(k-1)(n-k)+1$. We state it as a Proposition.
\begin{prop}\label{sharpness-of-thm1.1-angular}
    Suppose $k(n-k-1)+1\leq t\leq (k-1)(n-k)+1$. There exists a $t$-dimensional Frostman measure $\mu$ on $G(n,k)$ such that the following holds. We can find functions $f_{\delta}$ such that if $s<(k(n-k-1)+1-t)/2$,
    \begin{align*}
        \frac{\|\mathcal{M}^{n,k}(f_{\delta})\|_{L^2(\mu)}}{\|f_{\delta}\|_{H^s(\R^n)}}\to\infty,\;\delta\to0.
    \end{align*}
\end{prop}
\begin{proof}
    We first construct the measure $\mu$. We work in local coordinates for $G(n,k)$. We consider $k$-subspaces of the form 
\begin{align*}
    V_A=\{(y,Ay):y\in \R^k\},\;A\in\R^{(n-k)\times k}.
\end{align*}
Suppose $t=k(n-k-1)+r+\tau$, where $r$ is an integer with $1\leq r\leq k-1$ and $\tau\in[0,1)$. We consider matrices of the form 
\[
\begin{pmatrix}
\vphantom{\begin{matrix}1\\1\\1\\1\\1\end{matrix}}
\scalebox{1}{$B$}\\
a_1\quad\cdots\quad a_r\quad b\quad 0\quad\cdots\quad 0
\end{pmatrix}_{(n-k)\times k}.
\]
where $B\in[0,1]^{(n-k-1)\times k}$, $a=(a_1,\dots,a_r)\in[0,1]^r$ and $b\in E_\tau$ for some $\tau$-dimensional sets $E_\tau$ that admits a Ahlfors-David regular measure $\nu$ and $0\in\supp\nu $. For $x\in\R^n$, we write $x=(u,v,w,z)\in\R^{r}\times\R^{k-r}\times\R^{n-k-1}\times\R$.
And define 
\begin{align*}
    V_{B,a,b}=\{(u,v,B(u,v),a\cdot u+bv_1):u\in\R^r,v\in\R^{k-r}\}.
\end{align*}
We have 
\begin{align*}
    \dim \{V_{B,a,b}:B\in[0,1]^{(n-k-1)\times k}, a=(a_1,\dots,a_r)\in[0,1]^r,b\in E_\tau\}=t.
\end{align*}
Let $m$ be the Lebesgue measure on $B,a$, and let $\mu$ be the pushforward measure of $m\times \nu$. $\mu$ is clearly a $t$-dimensional Frostman measure on $G(n,k)$.

We turn to the construction of the functions $f_{\delta}$. The idea is to choose a function supported in a small region that intersects a large family of the $k$-planes $V_{B,a,b}$ in relatively large sets. We consider the subspace 
\begin{align*}
    W=\{(0,v,w,0):v\in\R^{k-r},\;w\in\R^{n-k-1}\}.
\end{align*}
We have $\dim V_{B,a,0}\cap W=k-r$ for all $B,a$. Then we take the region
\begin{align*}
    E_{\delta}=\{(u,v,w,z):|u|,|z|\lesssim  \delta,|v|,|w|\lesssim 1\}.
\end{align*}
 It is more convenient to choose smooth cut-off functions when computing the Sobolev norm. We choose two smooth cut-off functions $\chi\in C^{\infty}_0(\R^{r}\times \R)$ and $\eta\in C^{\infty}_0(\R^{k-r}\times\R^{n-k-1})$ such that $\chi,\eta\geq0$ and equals to $1$ near $0$. We take
\begin{align*}
    f_{\delta}(u,v,w,z)=\chi(u/\delta,z/\delta)\cdot\eta(v,w).
\end{align*}
We can choose suitable $\chi$, $\eta$ such that $f_{\delta}=1$ on $E_{\delta}$ and $\supp f\subset B(0,1)$. When $|b|\lesssim \delta$, we have $|a\cdot u+bv_1|\lesssim \delta$. Then 
\begin{align*}
    \mathcal{M}^{n,k}(f_{\delta})(V_{B,a,b})\gtrsim&\int\int f_{\delta}(u,v,B(u,v),a\cdot u+bv_1)d\H^r(u)d\H^{k-r}(v)\\
    \gtrsim&\int_{|u|\lesssim\delta}\int_{|v|\lesssim1}d\H^r(u)d\H^{k-r}(v)\\
    \gtrsim&\delta^r.
\end{align*}
Hence 
\begin{align*}
    \|\mathcal{M}^{n,k}(f_{\delta})\|_{L^2(\mu)}\gtrsim \nu(\{|b|\lesssim\delta\})^{1/2}\cdot\delta ^r\gtrsim \delta^{r+\tau/2}.
\end{align*}
Here $\nu(\{|b|\lesssim\delta\})\gtrsim \delta ^{\tau}$ since $\nu$ is Ahlfors-David regular and $0\in\supp\nu$. By a direct computation, when $-(r+1)/2<s\leq 0$,
\begin{align*}
    \|f_{\delta}\|_{H^s(\R^n)}\lesssim\delta^{(r+1)/2-s}.
\end{align*}
So
\begin{align*}
    \frac{\|\mathcal{M}^{n,k}(f_{\delta})\|_{L^2(\mu)}}{\|f_{\delta}\|_{H^s(\R^n)}}\gtrsim\delta^{(r+\tau-1)/2+s}=\delta^{(t-k(n-k-1)-1)/2+s}.
\end{align*}
When $s<(k(n-k-1)+1-t)/2$, the right-hand side tends to infinity.
\end{proof}
\subsection{Obstruction when $2k\le n$}\label{Obstruction when 2k_leq_n}
In this section, we show that our method fails when $2k \leq n$.
Our argument relies on Lemma \ref{Radial-operator}. That lemma requires $\beta>n/2-k\geq0$; when $2k \leq n$, we cannot obtain an $L^2$ bound.
To confirm this obstruction, take $f = \mathbf{1}_{\{2\delta \leq |x| \leq 1\}} \cdot |x|^{-k},\delta>0$. This example shows that $p>n/k$ is necessary, regardless of the underlying measure $\mu$. Consequently, a uniform $L^2$-estimate is impossible whenever $2k \leq n$.

When $p=2$, by Lemma \ref{Borel-selection} and Lemma \ref{translation-remove}, we can reduce $L^2$-estimate of $\mathcal{M}^{n,k}_{\delta}$ to $L^2$-estimate of $C_{\delta}$. But when $p>2$, $L^p$-estimate of $C_{\delta}$ cannot imply $L^p$-estimate of $\mathcal{M}^{n,k}_{\delta}$. More precisely, the following examples illustrate this.
\begin{prop}\label{translation-not-remove}
    For $2k\leq n$, there exists a $k(n-k-1)+1$-dimensional Frostman measure $\mu$ such that if $\supp f\subset B(0,1)$,
\begin{align*}
    \|C_{\delta}(f)\|_{L^p(\mu)}\lesssim \|f\|_{L^p(\R^n)},\;\forall p>n/k.
\end{align*}
But we can find functions $g_{\delta}$ such that for all $1\leq p<\infty$,
\begin{align*}
    \frac{\|\mathcal{M}^{n,k}_{\delta}(g_{\delta})\|_{L^p(\mu)}}{\|g_{\delta}\|_{L^p(\R^n)}}\to+\infty,\;\delta\to0.
\end{align*}
\end{prop}
That is, when $2k\leq n$, to estimate the $L^p$-norm of  $\mathcal{M}^{n,k}_{\delta}$, we cannot ignore the effect of the translation $V\mapsto x_V$.  
\begin{proof}
We first construct the measure $\mu$. The construction is similar to that in Proposition \ref{sharpness-of-thm1.1-angular}. We work in local coordinates for $G(n,k)$.
For $x\in\R^n$, we write $x=(y,s,z,t)\in\R^{k-1}\times \R\times\R^{n-k-1}\times \R$. 
We consider the subspaces of the form
\begin{align*}
    V_{A,B,C}=\{(y,s,Ay+Bs,Cs):y\in\R^{k-1},s\in\R\},
\end{align*}
where $A\in\R^{(n-k-1)\times(k-1)}$, $B\in\R^{n-k-1}$, $C\in\R$. 

Let $m$ be the Lebesgue measure on $\{(A,B,C):\|A\|,\|B\|,\|C\|\leq1\}$, and let $\mu$ be the pushforward measure of $m$ under the map $(A,B,C)\mapsto V_{A,B,C}$. Clearly $\mu$ is a $k(n-k-1)+1$-dimensional Frostman measure.

Then we construct the functions $g_{\delta}$. Let $K$ be a compact planar Kakeya set with zero Lebesgue measure, and let 
\begin{align*}
    E=\{(y,s,z,t):|y|\leq 1,|z|\leq 2,(s,t)\in K\}.
\end{align*}
By Fubini's theorem, $|E|=0$.
For $|C|\leq 1$, by construction of $K$, there exist $s_C$, $t_C$ such that
\begin{align*}
    \{(s_C+\tau,t_C+C\tau):|\tau|\leq 1/4\}\subset K.
\end{align*}
Then the set
\begin{align*}
    \{(y,s_C+\tau,Ay+B\tau,t_C+C\tau):|y|\leq 1,|\tau|\leq 1/4\}
\end{align*}
is contained in the intersection of $E$ and a translate of $V_{A,B,C}$. Let 
\begin{align*}
    g_{\delta}=1_{\mathcal{N}_{\delta}(E)},
\end{align*}
we have 
\begin{align*}
    \frac{\|\mathcal{M}^{n,k}_{\delta}(g_{\delta})\|_{L^p(\mu)}}{\|g_{\delta}\|_{L^p(\R^n)}}\gtrsim |\mathcal{N}_{\delta}(E)|^{-1/p}.
\end{align*}
And since $|E|=0$, $\lim_{\delta\to0}|\mathcal{N}_{\delta}(E)|^{-1/p}=+\infty$.

Finally, we show the $L^p$-estimate for $C_{\delta}$. It suffices to consider $C$. Since $\|A\|\leq1, \|B\|\leq 1,\|C\|\leq 1$, the Jacobian of the map $(y,s)\mapsto (y,s,Ay+Bs,Cs)$ is bounded. We have 
\begin{align}
    C(f)(V_{A,B,C})\sim\int f(y,s,Ay+Bs,Cs)dyds.
\end{align}
Let $x=(y,s)$ and $M=(A,B)$.
Let $F(x)=f(x,Mx,Cs)$, and let $w(x)=|x|^{n-k-1}|s|$. We choose this weight to cancel the Jacobian determinant that appears in integrations on $M,C$. We will see this later. By H\"older's inequality,
\begin{align*}
    \left(\int|F(x)|dx\right)^p\leq \left(\int |F(x)|^pw(x)dx\right) \left(\int_{|x|\leq 1} w(x)^{-1/(p-1)}dx\right)^{p-1}.
\end{align*}
For the second integral, we consider the polar coordinates: $x=r\theta, s=r\theta'$, where $\theta'$ is the last coordinate of $\theta$. Then 
\begin{align*}
    \int_{|x|\leq 1} w(x)^{-1/(p-1)}dyds\leq \int_0^1r^{k-1-(n-k)/(p-1)}\int_{S^{k-1}}|\theta'|^{-1/(p-1)}d\sigma^{k-1}(\theta)dr.
\end{align*}
Working in local coordinates, we see that $\int_{S^{k-1}}|\theta'|^{-1/(p-1)}d\sigma^{k-1}(\theta)\lesssim_{p,k}1$ if $p>2$. And by a direct computation $\int_0^1r^{k-1-(n-k)/(p-1)}dr\lesssim_{k,p}1$ if $p>n/k$. Hence
\begin{align}
    C(f)(V_{A,B,C})^p\lesssim_{k,p}\int |f(x,Mx,Cs)|^pw(x)dx.
\end{align} Now we integrate $M, C$. Fix $x$. We choose $j$ such that $|x_j|\gtrsim |x|$. Let $M^j$ be the $j$-column of $M$. For fixed $x$, the map $M^j\mapsto z=x_jM^j$ is a linear transformation on $\R^{n-k-1}$ with determinant $|x_j|^{n-k-1}$. Thus 
\begin{align*}
    \int|f(x,Mx,Cs)|^pdM^j\lesssim&|x|^{-(n-k-1)}\int_{|z|\lesssim 1} |f(x,z+\sum_{i\neq j}M^ix_i,Cs)|^pdz\\
    \leq& |x|^{-(n-k-1)}\int |f(x,z,Cs)|^pdz.
\end{align*}
Hence 
\begin{align*}
    \int|f(x,Mx,Cs)|^pdM\lesssim|x|^{-(n-k-1)}\int |f(x,z,Cs)|^pdz.
\end{align*}
Integrating over $C$, we have 
\begin{align*}
    \int\int|f(x,Mx,Cs)|^pdMdC\lesssim&|x|^{-(n-k-1)}|s|^{-1}\int |f(x,z,t)|^pdzdt\\
    =&w(x)^{-1}\int |f(x,z,t)|^pdzdt.
\end{align*}
Finally, multiply it by $w(x)$ and integrate over $x$,  we get 
\begin{align}
     \int \int\int C(f)(V_{A,B,C})^pdAdBdC\lesssim_{k,p}\int |f(x,z,t)|^pdxdzdt.
\end{align}
\end{proof}

\section{Maximal Operators, $2k\leq n$}
\label{maximal-function-2k-leq-n}
In this section, we prove Theorem \ref{Kakeya-imply-(n,k,t)}. It suffices to consider nonnegative functions
$f\in C_c^\infty(B(0,1))$.
The general case follows by the same approximation
and Fatou argument as in Lemma \ref{delta-to-delta-free}.
Suppose $n,k$ are positive integers with $2\leq k\leq n-1$. Let $t>k(n-k-1)+1$. Suppose that for some $0<\epsilon<t-k(n-k-1)-1$ and some $2\leq p<\infty$, we have 
    \begin{align}\label{maximal_hypothesis}
        \|\mathcal{M}^{n-k+1,1}_{\delta}(g)\|_{L^p(G(n-k+1,1))}\lesssim\delta^{-\epsilon/p}\|g\|_{L^p(\R^{n-k+1})},\;\supp g\subset B(0,1).
    \end{align}
    
Suppose $f$ is measurable and $\supp f\subset B(0,1)$. We can assume that $f\geq 0$. We consider the Littlewood-Paley decomposition:
\begin{align*}
    f=\sum_{j\geq 0}P_jf=\sum_{j\geq 0}f_j,
\end{align*}
where $\supp\widehat{f_0}\subset\{|\xi|\leq 2\}$ and $\supp\widehat{f_j}\subset\{2^{j-1}\leq|\xi|\leq 2^{j+1}\}$ for $j\geq 1$.
To use the Fourier support property of 
$f_j$, we need to consider the signed maximal operators, where the absolute value is taken outside the integral:
\begin{align*}
    \tilde{\mathcal{M}}^{n,k}_{\delta}(g)(V)=\frac{1}{\delta^{n-k}}\cdot\sup_{x\in \R^n}\left|\int_{\mathcal{N}_{\delta}(x+V)}g\right|.
\end{align*}
and 
\begin{align*}
    \tilde{\mathcal{M}}^{n,k}(g)(V)=\sup_{x\in \R^n}\left|\int_{x+V}g\right|.
\end{align*}
Since $f\geq 0$, we have 
\begin{align*}
    \mathcal{M}^{n,k}(f)(V)\leq \sum_{j\geq0}\tilde{\mathcal{M}}^{n,k}(f_j)(V).
\end{align*}

The following lemma reduces Frostman measure estimates to Haar measure estimates. We use the property that $f_j=P_j f$ and $\supp f\subset B(0,1)$. For the original maximal operator, it follows from a simple geometric observation. But for the signed maximal operator, we need a Fourier-analytic argument based on frequency localization in the directional variables.
\begin{lem}
    \begin{align}
        \|\tilde{\mathcal{M}}^{n,k}(f_j)\|_{L^p(\mu)}^p\lesssim_N 2^{j(k(n-k)-t)}\|\tilde{\mathcal{M}}^{n,k}(f_j)\|_{L^p(G(n,k))}^p+2^{-jN}\|f\|_{L^p(\R^n)}^p.
    \end{align}
\end{lem}
\begin{proof}
     It suffices to show it in local coordinates.
     We write $\R^n=\R^k\times \R^{n-k}$. For $A\in \R^{(n-k)\times k}$, we define
    \begin{align*}
        V_{A}=\{(y,Ay):y\in \R^k\}.
    \end{align*}
    
    We work in local coordinates $\{A\in\R^{(n-k)\times k}:\|A\|\leq 1\}$, and it is similar for other coordinates. In this coordinate, the invariant measure $\gamma_{n,k}$ is equivalent to the Lebesgue measure. Let $\nu$ denote the restriction of $\mu$ on this coordinate chart.
    When $\|A\|\lesssim 1$, we have 
    \begin{align*}
        \tilde{\mathcal{M}}^{n,k}(f_j)(V_A)\sim \sup_{b\in \R^{n-k}}\left|\int_{\R^k}f_j(y,Ay+b)d\H^k(y)\right|.
    \end{align*}
    
    Choose $\chi\in C^{\infty}_0(B(0,3))$ such that $\chi=1$ on $B(0,2)$.
 For $b\in\R^{n-k}$, we define 
    \begin{align*}
        P_{j,b}(A)=\int_{\R^k}\chi(y)f_j(y,Ay+b)d\H^k(y).
    \end{align*}
    Since $f_j=P_jf$ with $\supp f\subset B(0,1)$, we have 
    \begin{align*}
        \left|\int_{\R^k}f_j(y,Ay+b)d\H^k(y)-P_{j,b}(A)\right|\lesssim_N 2^{-jN}\|f\|_{L^p(\R^n)},\;\forall N\geq 1.
    \end{align*}
    
    Now we consider $P_{j,b}$.
    Using Fourier inversion on $f_j$, we have 
\begin{align*}
    P_{j,b}(A)=\int\int\chi(y)\hat{f_j}(\xi_1,\xi_2)e^{2\pi i(y\cdot \xi_1+b\cdot \xi_2)}e^{2\pi iAy\cdot \xi_2}d\xi dy.
\end{align*}
Since $Ay\cdot\xi_2=\langle A ,\xi_2y^T\rangle$ and $|\xi_2y^T|\leq 3\cdot2^{j+1}$,
\begin{align}
    \supp \widehat{P_{j,b}}\subset B(0, 3\cdot2^{j+1}).
\end{align}
Then we have 
\begin{align*}
    |P_{j,b}(A)|^p\lesssim_N 2^{jk(n-k)}\int_{\R^{(n-k)\times k}}(1+2^j\|A-B\|)^{-N}|P_{j,b}(B)|^pdB, \;\forall N\geq 1.
\end{align*}
Taking the supremum in $b$, we get the local constancy
\begin{align}
    &\sup_{b\in \R^{n-k}}|P_{j,b}(A)|^p\\
    \lesssim_N &2^{jk(n-k)}\int_{\R^{(n-k)\times k}}(1+2^j\|A-B\|)^{-N}\sup_{b\in \R^{n-k}}|P_{j,b}(B)|^pdB, \;\forall N\geq 1.\nonumber
\end{align}

We need to turn back to $\tilde{\mathcal{M}}^{n,k}(f_j)$. When $\|A\|\leq 1$, we have  
\begin{align*}
    \int_{\|B\|\geq 2}(1+2^j\|A-B\|)^{-N}\sup_{b\in \R^{n-k}}|P_{j,b}(B)|^pdB\lesssim 2^{j(n-N)}\|f\|_{L^p(\R^n)}^p.
\end{align*}
By the Frostman property of $\nu$, when $N>t$,
\begin{align}
    \int_{\R^{(n-k)\times k}}(1+2^j\|A-B\|)^{-N}d\nu(A)\lesssim 2^{-jt}.
\end{align}
So 
\begin{align*}
    &\int \sup_{b\in \R^{n-k}}|P_{j,b}(A)|^pd\nu(A)\\\lesssim&_N 2^{jk(n-k)}\int_{\R^{(n-k)\times k}}\int(1+2^j\|A-B\|)^{-N}d\nu(A)\sup_{b\in \R^{n-k}}|P_{j,b}(B)|^pdB\\
    \lesssim&_N 2^{j(k(n-k)-t)}\int_{\|B\|\leq 2}\sup_{b\in \R^{n-k}}|P_{j,b}(B)|^p  dB+2^{j(n+k(n-k)-N)}\|f\|_{L^p(\R^n)}^p.
\end{align*}
Recall that  we have 
\begin{align*}
    \tilde{\mathcal{M}}^{n,k}(f_j)(V_A)\lesssim_N \sup_{b\in\R^{n-k}}|P_{j,b}(A)|+2^{-jN}\|f\|_{L^p(\R^n)}
\end{align*}
and
\begin{align*}
    \sup_{b\in\R^{n-k}}|P_{j,b}(B)|\lesssim_N\tilde{\mathcal{M}}^{n,k}(f_j)(V_B)+2^{-jN}\|f\|_{L^p(\R^n)},\;\|B\|\leq 2.
\end{align*}

Hence 
\begin{align*}
    &\int |\tilde{\mathcal{M}}^{n,k}(f_j)(V_A)|^pd\nu(A)\\
    \lesssim&_N2^{j(k(n-k)-t)}\int_{\|B\|\leq 2}|\tilde{\mathcal{M}}^{n,k}(f_j)(V_B)|^p dB+2^{j(n+k(n-k)-N)}\|f\|_{L^p(\R^n)}^p.
\end{align*}
\end{proof}

Now it suffices to estimate $\|\tilde{\mathcal{M}}^{n,k}(f_j)\|_{L^p(G(n,k))}$. By the locally constant property of $f_j$, we have 
\begin{align*}
    \tilde{\mathcal{M}}^{n,k}(f_j)\lesssim\tilde{\mathcal{M}}^{n,k}_{c2^{-j}}(f_j)
\end{align*}
for some small $c$.  We prove the following estimate for $\delta=2^{-j}$.
The same argument applies to $\delta=c2^{-j}$ for any
fixed $c>0$, with constants depending on $c$.

Now it suffices to consider $\|\tilde{\mathcal{M}}^{n,k}_{2^{-j}}(f_j)\|_{L^p(G(n,k))}^p$. We claim that 
\begin{align}\label{maximal-estimate-fourier-support}
    \|\tilde{\mathcal{M}}^{n,k}_{2^{-j}}(f_j)\|_{L^p(G(n,k))}^p\lesssim 2^{-j(k-1-\epsilon)}\|f\|^p_{L^p(\R^n)}.
\end{align}
Fix $j$, and let $\delta=2^{-j}$, $g=f_j$. For $V\in G(n,k)$, we can write $V=L\oplus U$, for $U\in G(n,k-1)$ and $L\in G(U^{\bot},1)$. By Fubini's theorem, we have  
    \begin{align}
        \tilde{\mathcal{M}}^{n,k}_{\delta}(g)(V)=\tilde{\mathcal{M}}^{U^{\bot},1}_{\delta}(\mathcal{X}_{k-1}(g)(U,\cdot))(L),
    \end{align}
where 
\begin{align*}
    \mathcal{X}_{k-1}(g)(U,y)=\int_U g(y+z)d\H^{k-1}(z),\;y\in U^{\bot}
\end{align*}
and 
\begin{align*}
    \tilde{\mathcal{M}}^{U^{\bot},1}_{\delta}(h)(L)=\frac{1}{\delta^{n-k}}\cdot\sup_{x\in U^{\bot}}\left|\int_{\mathcal{N}_{\delta}(x+L)}h\right|.
\end{align*}
Recall that we have a recursive formula for $\gamma_{n,k}$:
    \begin{align*}
        \gamma_{n,k}(A)=\int_{G(n,k-1)}\gamma_{U^{\bot},1}(\{L\subset U^{\bot}: L+U\in A\})d\gamma_{n,k-1}(U),
    \end{align*}
    then we can rewrite $\|\tilde{\mathcal{M}}^{n,k}_{\delta}(g)\|_{L^p(G(n,k))}^p$:
    \begin{align}\label{induction-identity}
        &\|\tilde{\mathcal{M}}^{n,k}_{\delta}(g)\|_{L^p(G(n,k))}^p\\
        =&\int_{G(n,k-1)}\int_{G(U^{\bot},1)}|\tilde{\mathcal{M}}^{U^{\bot},1}_{\delta}(\mathcal{X}_{k-1}(g)(U,\cdot))(L)|^p d\gamma_{U^{\bot},1}(L)d\gamma_{n,k-1}(U).\nonumber
    \end{align}
    Recall that $g=f_j$, by the rapid decay of $g$ outside $B(0,2)$, we may
    apply (\ref{maximal_hypothesis}) up to an error term. 
     
    \begin{align*}
        &\|\tilde{\mathcal{M}}^{n,k}_{\delta}(g)\|_{L^p(G(n,k))}^p\\
        \lesssim& \delta^{-\epsilon}\int_{G(n,k-1)}\int_{U^{\bot}}|\mathcal{X}_{k-1}(g)(U,y)|^pd\H^{n-k+1}(y)d\gamma_{n,k-1}(U)+2^{-jN}\|f\|_{L^p(\R^n)}^p.
    \end{align*}

To estimate the $L^p$-norm of $\mathcal{X}_{k-1}(g)$, we need to use the property that $g=f_j=P_jf$ with $\supp f\subset B(0,1)$.
Note that for fixed $U$, the Fourier transform of the function $\mathcal{X}_{k-1}(P_jf)(U,y)$ is supported on the annulus $\{2^{j-1}\leq |\xi|\leq 2^{j+1}\}$, we have 
\begin{align*}
    \int_{G(n,k-1)}\int_{U^{\bot}}|\mathcal{X}_{k-1}(P_jf)(U,y)|^2d\H^{n-k+1}(y)d\gamma_{n,k-1}(U)&\sim\int |\widehat{P_jf}(\xi)|^2|\xi|^{-(k-1)}d\xi\\
    &\lesssim 2^{-(k-1)j}\|f\|_{L^2(\R^n)}^2.
\end{align*}
Since $\supp f\subset B(0,1)$, we have the trivial $L^{\infty}$-estimate
\begin{align*}
    \|\mathcal{X}_{k-1}(P_jf)\|_{L^{\infty}}\lesssim\|f\|_{L^{\infty}(\R^n)}.
\end{align*}
Interpolation between $L^2$ and $L^{\infty}$ estimates, we have 
\begin{align*}
    \int_{G(n,k-1)}\int_{U^{\bot}}|\mathcal{X}_{k-1}(P_jf)(U,y)|^pd\H^{n-k+1}(y)d\gamma_{n,k-1}(U)\lesssim 2^{-(k-1)j}\|f\|_{L^p(\R^n)}^p.
\end{align*}
Hence 
\begin{align*}
    \|\tilde{\mathcal{M}}^{n,k}_{2^{-j}}(f_j)\|_{L^p(G(n,k))}^p\lesssim 2^{-j(k-1-\epsilon)}\|f\|^p_{L^p(\R^n)}.
\end{align*}

We return to $\mathcal{M}^{n,k}(f)$. We have 
\begin{align*}
    \|\mathcal{M}^{n,k}(f)\|_{L^p(\mu)}\lesssim& \sum_{j\geq 0}2^{j(k(n-k)-t)/p}\|\tilde{\mathcal{M}}^{n,k}(f_j)\|_{L^p(G(n,k))}+\|f\|_{L^p(\R^n)}\\
    \lesssim& \sum_{j\geq 0}2^{j(k(n-k)-t)/p}\|\tilde{\mathcal{M}}^{n,k}_{c2^{-j}}(f_j)\|_{L^p(G(n,k))}+\|f\|_{L^p(\R^n)}\\
    \lesssim& \sum_{j\geq 0}2^{j(k(n-k-1)+1-t+\epsilon)/p}\|f\|_{L^p(\R^n)}
    \\\lesssim&\|f\|_{L^p(\R^n)}.
\end{align*}
Here the geometric series $\sum_{j\geq 0}2^{j(k(n-k-1)+1-t+\epsilon)/p}$ converges since  $t>k(n-k-1)+1+\epsilon$.

\section{$(n,k,t)$-sets}\label{restricted-(n,k)-section} 
In this section, we first show that  $L^p$-$L^q$ estimates of $\mathcal{M}^{n,k}$ for some $1\leq p,q<\infty$ imply positive measure of $(n,k,t)$-sets, and then explain why $t>k(n-k-1)+1$ is necessary by explicit examples.

\subsection{From maximal estimates to positive measure}\label{proof-of-main-maximal}
Suppose that for some $1\leq p,q<\infty$, we have 
\begin{align}\label{Lp-Lq-estimate}
    \|\mathcal{M}^{n,k}(f)\|_{L^q(\mu)}\lesssim \|f\|_{L^{p}(\R^n)}
\end{align}
for any $t$-dimensional Frostman measure $\mu$. Then we show that $(n,k,t)$-sets have positive measure. The idea is a simple pigeonholing argument.

Suppose $B$ is an $(n,k,t)$-set associated to $\mathcal{V}$ with $\dim\mathcal{V}= t>k(n-k-1)+1$. By slightly decreasing $t$ if necessary, we can assume
that $t>k(n-k-1)+1$ and $\mathcal{H}^t_\infty(\mathcal{V})>0$.

By definition of Lebesgue measure, we can find a $G_{\delta}$-set $G$ such that $B\subset G$ and $|B|=|G|$. So we can assume that $B$ is a $G_{\delta}$-set. We write 
\begin{align*}
    B=\bigcup_{m\geq 1}B\cap B(0,m)=\bigcup_{m\geq 1}B_m.
\end{align*}
By definition, we have
\begin{align*}
    \mathcal{M}^{n,k}(1_{B})(V)\gtrsim 1.
\end{align*}
Then there exists the smallest $m=m(V)$ such that
\begin{align*}
    \mathcal{M}^{n,k}(1_{B_m})(V)\gtrsim 1/m^2.
\end{align*}
By definition, we have
\begin{align*}
    \mathcal{V}\subset\bigcup_{m\geq1}G_m,
\end{align*}
where
\begin{align*}
    G_m=\{V\in G(n,k):\mathcal{M}^{n,k}(1_{B_m})(V)\gtrsim 1/m^2\}.
\end{align*}
Then by the pigeonhole principle and the countable subadditivity of Hausdorff content, there exists $G_m$ such that $\H^t_{\infty}(G_m)>0$. Since $B_m$ is a bounded $G_{\delta}$ set, the set $G_m$ is analytic. By Frostman's lemma, there exists a $t$-dimensional Frostman measure $\mu$ supported on $G_m$. We apply the estimate   (\ref{Lp-Lq-estimate}) to $f=1_{B_m}$ and the measure $\mu$, we have 
\begin{align*}
     0<1/m^{2q}\lesssim\|\mathcal{M}^{n,k}(f)\|_{L^q(\mu)}^q\lesssim
        \|f\|_{L^p}^q=|B_m|^{q/p}\leq |B|^{q/p}.
\end{align*}
\subsection{$(n,k,k(n-k-1)+1)$-sets with Lebesgue measure zero}\label{examples}
We claim that there exists a $(n,k,k(n-k-1)+1)$-set with Lebesgue measure zero. We give two constructions.

The first approach considers all $k$-dimensional subspaces contained in an \((n-1)\)-dimensional subspace, and then applies a one-dimensional perturbation of measure zero. 

We choose a compact subset $E\subset\R$ such that $\dim E=1$ and $|E|=0$. We consider matrices of the form 
\[
\begin{pmatrix}
A\\
0\ \cdots\ 0\ \ a
\end{pmatrix}
\qquad
\]
with $A\in\mathbb R^{(n-k-1)\times k},\;a\in E$. 
For each $A\in\mathbb R^{(n-k-1)\times k},\;a\in E$, we define 
    \begin{align*}
        V_{A,a}=\left\{(x,Ax,ax_k):x=(x_1,\cdots,x_k)\in\R^k\right\}
    \end{align*}
Then we take the collection
    \begin{align*}
        \mathcal{V}=\{V_{A,a}:A\in\mathbb R^{(n-k-1)\times k},\;a\in E\}
    \end{align*}
     It is clear that $\dim \mathcal{V}=k(n-k-1)+1$. We consider the set
    \begin{align*}
        B=\bigcup_{V\in\V}V.
    \end{align*}
     By definition, $B$ is an $(n,k,k(n-k-1)+1)$-set.
    By Fubini's theorem, 
    \begin{align*}
        |B|&=\int_{\R^k}\H^{n-k}(B_x)d\H^k(x),
    \end{align*}
    where $B_x$ is the $x$-slice of $B$, given by 
    \begin{align*}
        B_x
        \cong\left\{(Ax,ax_k):A\in\mathbb R^{(n-k-1)\times k},\;a\in E\right\}\subset \R^{n-k-1}\times x_kE.
    \end{align*}
    Since $|E|=0$, by Fubini's theorem, $\H^{n-k}(B_x)=0$. Then $|B|=0$,
     $B$ is a $(n,k,k(n-k-1)+1)$-set with Lebesgue measure zero.

The second approach considers a Cartesian product construction based on classical planar Kakeya sets. We take a planar Kakeya set $K\subset \R^2$ with zero measure that contains a length-two line segment in every direction.  And  we take
\begin{align*}
    B=K\times\R^{n-2}.
\end{align*}
By Fubini's theorem, $|B|=0$. It suffices to check $B$ is a $(n,k,k(n-k-1)+1)$-set.
We take $U=\{0\}\times\R^{n-2}$ and consider 
\begin{align*}
    \Sigma=\{V\in G(n,k):\dim V\cap U=k-1\}.
\end{align*}
The set $\Sigma$ is a smooth submanifold of dimension $k(n-k-1)+1$. We claim that $B$ contains a translation of the unit disk of every element in $\Sigma$. 

Suppose $V\in \Sigma$. For convenience, we consider the orthogonal decomposition
\begin{align*}
    V=(V\cap U)\oplus L.
\end{align*}
We take a unit vector $l\in L$, then the unit disk of $V$ is given by 
\begin{align*}
    D_V=\{u+sl:u\in V\cap U,\;\|u\|^2+s^2\leq 1\}.
\end{align*}
Its projection onto $U^{\bot}$ is given by 
\begin{align*}
    \pi_{U^{\bot}}(D_V)=\{s\pi_{U^{\bot}}l:|s|\leq 1\}.
\end{align*}
This is a line segment with length $\leq 2$. By definition of $K$, we can choose a translation $x_V\in \R^2$ such that 
\begin{align*}
    \pi_{U^{\bot}}(D_V)+x_V\subset K\times\{0\}.
\end{align*}
Then 
\begin{align*}
    D_V+(x_V,0)\subset B.
\end{align*}

\bibliographystyle{abbrv}
    \bibliography{bibtex}

\end{document}